\documentclass[11pt]{article}

\usepackage[T1]{fontenc}
\usepackage[utf8]{inputenc}
\usepackage{lmodern}
\usepackage{amsmath,amssymb,amsthm,mathtools}
\usepackage{enumitem}
\usepackage{booktabs}
\usepackage{geometry}
\usepackage{hyperref}
\usepackage{xcolor}

\hypersetup{
  colorlinks=true,
  linkcolor=blue,
  citecolor=blue,
  urlcolor=blue
}

\newtheorem{theorem}{Theorem}[section]
\newtheorem{proposition}[theorem]{Proposition}
\newtheorem{lemma}[theorem]{Lemma}
\newtheorem{corollary}[theorem]{Corollary}

\newtheorem{question}[theorem]{Question}
\theoremstyle{definition}
\newtheorem{definition}[theorem]{Definition}
\theoremstyle{remark}
\newtheorem{remark}[theorem]{Remark}
\newtheorem*{introtheorem}{\textbf{Theorem}}

\newcommand{\Fp}{\mathbb F_p}
\newcommand{\Ftwo}{\mathcal F_2}
\newcommand{\PP}{\mathbb P}
\newcommand{\UU}{U}
\newcommand{\Ubar}{\overline U}
\newcommand{\cHH}{\mathrm H}

\title{Higher Massey Products in Demu\v{s}kin Variations:\\
Support Blocks, One-Relator Reduction, and a Five-Fold Vanishing Case}
\author{Marina Palaisti}
\date{ }

\begin{document}
\maketitle

\begin{abstract}
Blumer and Quadrelli introduced a family $\Ftwo$ of two-relator pro-$p$ groups obtained from a Demu\v{s}kin group by imposing the commutativity of two generators which are not paired in the Demu\v{s}kin relation. They proved the triple and quadruple Massey vanishing properties and asked whether the same holds in every length. We introduce the $z$-profile
\(v_h=(\alpha_h(z_1),\alpha_h(z_2))\in\Fp^2\) of a defined $n$-fold Massey product. Definability forces $\det(v_h,v_{h+1})=0$, so the nonzero profile entries decompose into support blocks carrying projective directions in $\PP^1(\Fp)$. After the known endpoint reduction, profiles with exactly $r$ blocks are counted by $\binom{n-1}{2r}$, and $r$ blocks first occur in length $2r+1$. We also prove an all-length conditional reduction for Dwyer's lifting problem: if the added commuting relator can be made exact in a lift, then the remaining central defect of the Demu\v{s}kin relator can be removed by an endpoint correction. The correction preserves the power term $x_1^q$
for every parameter allowed in $\Ftwo$. In length five, this yields vanishing whenever all three interior profile vectors $v_2,v_3,v_4$ are nonzero, for every prime $p$. The resulting classification identifies the remaining five-fold support types and the compatibility mechanisms governing them.
\end{abstract}



\section{Introduction}

Demu\v{s}kin groups occupy a distinguished position in pro-$p$ group theory
and Galois cohomology. They are Poincar\'e duality pro-$p$ groups of
dimension two, their cup-product pairing is nondegenerate, and the maximal
pro-$p$ Galois group of a $p$-adic field containing the relevant roots of
unity is a Demu\v{s}kin group. Moreover, Demu\v{s}kin groups satisfy the
strong $n$-fold Massey vanishing property for every $n\geq3$; see, for
example, \cite{MinacTanMassey,BlumerCassellaQuadrelli}.

Blumer and Quadrelli recently introduced families of pro-$p$ groups which
are deliberately close to Demu\v{s}kin groups but fail $1$-cyclotomicity
\cite{BlumerQuadrelli}. The family relevant here consists of groups
\begin{equation}
G=
\left\langle
x_1,y_1,\ldots,x_d,y_d
\ \middle|\
x_1^q[x_1,y_1]\cdots[x_d,y_d]=1,\quad [z_1,z_2]=1
\right\rangle_{\widehat p},
\label{eq:presentation}
\end{equation}
where $d\geq2$, $q=p^k$ with
$k\in\{1,2,\ldots,\infty\}$, with $k\geq2$ if $p=2$, and with the
convention $p^\infty=0$. The generators $z_1,z_2$ are distinct and do not
form one of the paired sets $\{x_i,y_i\}$. We denote this family by
$\Ftwo$.

The groups in $\Ftwo$ retain several Galois-like properties, including
quadratic cohomology and low-length Massey vanishing, but they are not
maximal pro-$p$ Galois groups of fields containing a primitive $p$th root
of unity \cite{BlumerQuadrelli}. Blumer and Quadrelli proved the three-fold
and four-fold Massey vanishing properties and ask whether every
$G\in\Ftwo$ satisfies the $n$-fold Massey vanishing property for every
$n\geq3$ \cite[Question~5.13]{BlumerQuadrelli}.

The present paper develops two reductions of that question.

The first is combinatorial. For a defined product
\[
\langle\alpha_1,\ldots,\alpha_n\rangle\subseteq\cHH^2(G,\Fp)
\]
we attach the sequence
\[
v_h=(\alpha_h(z_1),\alpha_h(z_2))\in\Fp^2.
\]
The extra commuting relation contributes an independent alternating
cup-product component, and definability forces
\[
\det(v_h,v_{h+1})=0.
\]
Consequently consecutive nonzero profile vectors are projectively equal.
The nonzero positions split into maximal support blocks, each labelled by
a direction in $\PP^1(\Fp)$.

The second reduction is representation-theoretic. Dwyer's correspondence
turns a defined $n$-fold Massey product into a lifting problem
\[
G\longrightarrow \Ubar_{n+1}
\qquad\leadsto\qquad
G\longrightarrow \UU_{n+1},
\]
where $\Ubar_{n+1}=\UU_{n+1}/Z(\UU_{n+1})$. The two defining relators of
\eqref{eq:presentation} give two central defects. We prove the following
conditional all-length reduction: once the commuting defect has been
killed, the remaining Demu\v{s}kin defect can always be removed by an
endpoint correction. The $E_{2,n+1}$ correction appears already in the
final step of the low-length argument of Blumer--Quadrelli
\cite[proof of Proposition~5.12]{BlumerQuadrelli}; here its
power-preservation mechanism is isolated in a length-independent form.
Thus the two-relator lifting problem reduces to exactification of $[z_1,z_2]$.

This observation produces a concrete five-fold theorem. In length five,
if all three interior profile entries are nonzero, adjacent collinearity
forces a single projective direction. If that direction is a coordinate
axis, the product vanishes by an all-length theorem of Blumer--Quadrelli.
Otherwise, after replacing one of the two commuting matrices by an
auxiliary quotient of the form $C_2C_1^{-\kappa}$, the commuting problem
reduces to a normalized $\UU_6$ calculation. The calculation forces the
commuting relator to be exact, and the one-relator reduction then removes
the remaining Demu\v{s}kin defect.

The main results may be summarized as follows.

\begin{introtheorem}[Support-block hierarchy]
\label{thm:intro-block}
Let $G\in\Ftwo$, let $n\geq3$, and suppose
$\langle\alpha_1,\ldots,\alpha_n\rangle$ is defined. If the product is not
covered by the endpoint or coordinate-hyperplane vanishing theorems of
Blumer--Quadrelli, then:
\begin{enumerate}[label=\textup{(\roman*)}]
\item $v_1=v_n=0$;
\item every support block is contained in $\{2,\ldots,n-1\}$ and carries
one projective direction in $\PP^1(\Fp)$;
\item if there are $r$ support blocks, then
\[
r\leq\left\lfloor\frac{n-1}{2}\right\rfloor;
\]
\item the number of possible interior zero/nonzero patterns with exactly
$r$ blocks is
\[
\binom{n-1}{2r}.
\]
In particular, $r$ blocks first occur in length $2r+1$, where the unique
pattern is alternating.
\end{enumerate}
\end{introtheorem}

\begin{introtheorem}[One-relator reduction]
\label{thm:intro-one-relator}
Let $G\in\Ftwo$, let $n\geq3$, and let
$\langle\alpha_1,\ldots,\alpha_n\rangle$ be defined. Assume the product is
not already covered by the endpoint vanishing theorem. If a Dwyer
representation admits lifts of the generator matrices for which the
commuting relator $[z_1,z_2]$ is exact, then the product contains $0$.
\end{introtheorem}

\begin{introtheorem}[Full interior support in length five]
\label{thm:intro-five}
Let $G\in\Ftwo$. If
$\langle\alpha_1,\ldots,\alpha_5\rangle$ is defined and
\[
v_2\neq0,\qquad v_3\neq0,\qquad v_4\neq0,
\]
then the product contains $0$.
\end{introtheorem}

The five-fold status after the known endpoint reduction is summarized
in Table~\ref{tab:intro-five-status}. Coordinate-hyperplane instances are
already covered by Proposition~\ref{prop:hyperplane}; the table records
what remains for the indicated support pattern.

\begin{table}[ht]
\centering
\small
\begin{tabular}{ccl}
\toprule
interior support & blocks & status in this paper \\
\midrule
$\varnothing$ & $0$ & vanishes by coordinate-hyperplane reduction \\
$\{2\}$ & $1$ & unresolved singleton case \\
$\{3\}$ & $1$ & normalized commuting defect removable \\
$\{4\}$ & $1$ & unresolved singleton case \\
$\{2,3\}$ & $1$ & normalized commuting defect removable \\
$\{3,4\}$ & $1$ & normalized commuting defect removable \\
$\{2,3,4\}$ & $1$ & vanishes by Theorem~\ref{thm:intro-five} \\
$\{2,4\}$ & $2$ & separated two-block problem \\
\bottomrule
\end{tabular}
\caption{Length-five support types after endpoint reduction.}
\label{tab:intro-five-status}
\end{table}

Consequently, full five-fold vanishing for $\Ftwo$ is reduced here to the
six support types
\[
\{2\},\ \{3\},\ \{4\},\ \{2,3\},\ \{3,4\},\ \{2,4\}.
\]
For $\{3\},\{2,3\},\{3,4\}$ we identify normalized-model corrections; the
remaining issue is compatibility with the full Demu\v{s}kin presentation.

The arithmetic specialization at the end of the paper records a natural
field-parametrized subfamily obtained from maximal pro-$p$ Galois groups
of $p$-adic fields. This produces a natural field-parametrized subfamily of the Blumer–Quadrelli variations. For comparison, Maire--Min\'a\v{c}--Ramakrishna--T\^an
prove strong all-length vanishing for number fields $K$ with
$\zeta_p\notin K$ and $p$ odd \cite{MaireMinacRamakrishnaTan}.

\section{Preliminaries}

\subsection{The family \texorpdfstring{$\Ftwo$}{F2}}

Fix a prime $p$. Let $G\in\Ftwo$ have presentation
\eqref{eq:presentation}, with the parameter convention stated above, and put
\[
X_G=\{x_1,y_1,\ldots,x_d,y_d\}.
\]
For $i=1,2$, let $t_i$ denote the Demu\v{s}kin partner of $z_i$, so that
\[
\{z_i,t_i\}=\{x_{j_i},y_{j_i}\}.
\]

As computed by Blumer and Quadrelli
\cite[\S5.1]{BlumerQuadrelli}, the mod-$p$ cohomology satisfies
\[
\dim_{\Fp}\cHH^2(G,\Fp)=2.
\]
With respect to the standard dual basis of $\cHH^1(G,\Fp)$, one may
choose a basis $\{\eta,\zeta\}$ of $\cHH^2(G,\Fp)$ such that $\eta$
records the Demu\v{s}kin cup-product component and
\[
\zeta=z_1^*\smile z_2^*.
\]
Thus, for $\alpha,\beta\in\cHH^1(G,\Fp)$, if
\[
v(\alpha)=(\alpha(z_1),\alpha(z_2)),
\]
then the coefficient of $\zeta$ in $\alpha\smile\beta$ is
\begin{equation}
\omega_z(\alpha,\beta)
=
\alpha(z_1)\beta(z_2)-\alpha(z_2)\beta(z_1)
=
\det(v(\alpha),v(\beta)).
\label{eq:zdet}
\end{equation}
This is the only part of the cup-product formula needed for the
support-block reduction.

\subsection{Dwyer's criterion}

Throughout the paper, $\cHH^\bullet(G,\Fp)$ denotes continuous
cohomology with trivial discrete coefficients. The groups
$\UU_{n+1}=\UU_{n+1}(\Fp)$ and their quotients are finite $p$-groups,
viewed as profinite groups with the discrete topology; every
homomorphism of pro-$p$ groups appearing below is continuous.

The center of $\UU_{n+1}$ is
\[
Z(\UU_{n+1})
=
\{I_{n+1}+aE_{1,n+1}:a\in\Fp\},
\]
and we put
\[
\Ubar_{n+1}=\UU_{n+1}/Z(\UU_{n+1}).
\]

We use Dwyer's correspondence in the pro-$p$ setting; see
\cite{Dwyer} and \cite[\S2.3]{BlumerQuadrelli}.

\begin{proposition}[Dwyer]
\label{prop:dwyer}
Let $G$ be a pro-$p$ group and let
$\alpha_1,\ldots,\alpha_n\in\cHH^1(G,\Fp)$.
\begin{enumerate}[label=\textup{(\alph*)}]
\item The product $\langle\alpha_1,\ldots,\alpha_n\rangle$ is defined
if and only if there exists a continuous homomorphism
\[
\bar\rho:G\longrightarrow\Ubar_{n+1}
\]
whose $(h,h+1)$ entry is $\alpha_h$ for every $h$.
\item The product contains $0$ if and only if there exists a continuous
homomorphism
\[
\rho:G\longrightarrow\UU_{n+1}
\]
whose $(h,h+1)$ entry is $\alpha_h$ for every $h$.
\end{enumerate}
\end{proposition}

A homomorphism $\bar\rho$ as in part~(a) will be called a
\emph{Dwyer representation}. Definability implies
\begin{equation}
\alpha_h\smile\alpha_{h+1}=0,
\qquad h=1,\ldots,n-1.
\label{eq:adjcup}
\end{equation}
We will also use that if $\alpha_h=0$ for some $h$, then a defined
$n$-fold Massey product contains $0$; see
\cite[Lemma~2.5(a)]{BlumerQuadrelli}.

\subsection{The all-length reductions of Blumer--Quadrelli}

The following two results are Proposition~5.10 and Proposition~5.11 of
\cite{BlumerQuadrelli}. We state them in the form used below.

\begin{proposition}
\label{prop:endpoint}
Suppose $\langle\alpha_1,\ldots,\alpha_n\rangle$ is defined. If
\[
\alpha_1(z_i)\neq0
\qquad\text{or}\qquad
\alpha_n(z_i)\neq0
\]
for some $i\in\{1,2\}$, then the product contains $0$.
\end{proposition}

\begin{proposition}
\label{prop:hyperplane}
Suppose $\langle\alpha_1,\ldots,\alpha_n\rangle$ is defined. If, for
some fixed $i\in\{1,2\}$,
\[
\alpha_h(z_i)=0
\qquad\text{for every }h=1,\ldots,n,
\]
then the product contains $0$.
\end{proposition}

\begin{remark}
\label{rem:stronger-zero}
Proposition~\ref{prop:hyperplane} is stronger than the statement that
the product vanishes when both $z$-coordinates vanish for every entry.
It says that vanishing of only one fixed commuting coordinate along the
entire sequence is sufficient.
\end{remark}

Blumer and Quadrelli also prove that every defined product of length
three or four in $\Ftwo$ contains $0$ \cite{BlumerQuadrelli}.

\section{Support blocks and their hierarchy}

\begin{definition}[The $z$-profile]
For $\alpha=(\alpha_1,\ldots,\alpha_n)$ define
\[
v_h=v(\alpha_h)
=
(\alpha_h(z_1),\alpha_h(z_2))\in\Fp^2.
\]
The sequence
\[
v(\alpha)=(v_1,\ldots,v_n)
\]
is the \emph{$z$-profile}.
\end{definition}

\begin{definition}[Support block]
A \emph{support block} is a maximal interval
\[
I=[a,b]\cap\mathbb Z\subseteq\{1,\ldots,n\}
\]
such that $v_h\neq0$ for every $h\in I$. We write $b(v)$ for the number
of support blocks.
\end{definition}

\begin{lemma}
\label{lem:adjacent}
If $\langle\alpha_1,\ldots,\alpha_n\rangle$ is defined, then
\[
\det(v_h,v_{h+1})=0
\qquad
(h=1,\ldots,n-1).
\]
\end{lemma}

\begin{proof}
By \eqref{eq:adjcup},
$\alpha_h\smile\alpha_{h+1}=0$. The coefficient of $\zeta$ in this cup
product is, by \eqref{eq:zdet},
$\det(v_h,v_{h+1})$. Hence it is zero.
\end{proof}

\begin{proposition}
\label{prop:projective}
Let $I$ be a support block of a defined Massey product. Then there is a
unique direction
\[
\ell_I\in\PP^1(\Fp)
\]
such that $v_h\in\ell_I$ for every $h\in I$.
\end{proposition}

\begin{proof}
If $h,h+1\in I$, both vectors are nonzero and
Lemma~\ref{lem:adjacent} gives
$\det(v_h,v_{h+1})=0$. Thus they span the same one-dimensional
subspace. Induction along the interval proves the claim.
\end{proof}

\begin{definition}[Block word]
If the support blocks are $I_1<\cdots<I_r$ with directions
$\ell_1,\ldots,\ell_r$, the sequence
\[
\bigl((I_1,\ell_1),\ldots,(I_r,\ell_r)\bigr)
\]
is the \emph{block word} of the profile.
\end{definition}

\begin{theorem}
\label{thm:block-reduction}
Let $G\in\Ftwo$ and suppose
$\langle\alpha_1,\ldots,\alpha_n\rangle$ is defined. If the product is
not covered by Propositions~\ref{prop:endpoint} and
\ref{prop:hyperplane}, then:
\begin{enumerate}[label=\textup{(\roman*)}]
\item $v_1=v_n=0$;
\item $b(v)\geq1$;
\item every support block is contained in $\{2,\ldots,n-1\}$ and has a
single projective direction;
\item for each $i\in\{1,2\}$ there is an interior index
$h\in\{2,\ldots,n-1\}$ such that $\alpha_h(z_i)\neq0$.
\end{enumerate}
\end{theorem}

\begin{proof}
Part (i) is the negation of the hypothesis of
Proposition~\ref{prop:endpoint}. If $b(v)=0$, then both commuting
coordinates vanish everywhere, so Proposition~\ref{prop:hyperplane}
applies; hence (ii). Part (iii) follows from (i) and
Proposition~\ref{prop:projective}. Finally, if for one fixed
$i\in\{1,2\}$ every $\alpha_h(z_i)$ vanished, then
Proposition~\ref{prop:hyperplane} would apply. The endpoint coordinates
already vanish, so the required nonzero coordinate occurs in the
interior.
\end{proof}

The block count itself has a simple closed form.

\begin{theorem}[Block-count hierarchy]
\label{thm:block-count}
After endpoint reduction, let a zero/nonzero profile on the
$n-2$ interior positions have exactly $r\geq1$ support blocks. Then
\[
r\leq\left\lfloor\frac{n-1}{2}\right\rfloor.
\]
Moreover, the number of interior zero/nonzero patterns with exactly
$r$ support blocks is
\begin{equation}
N(n,r)=\binom{n-1}{2r}.
\label{eq:block-count}
\end{equation}
Consequently, $r$ support blocks first occur for $n=2r+1$, and at that
length the unique pattern is
\[
(0,\bullet,0,\bullet,\ldots,0,\bullet,0).
\]
\end{theorem}

\begin{proof}
Write $m=n-2$ for the number of interior positions. A binary word of
length $m$ with exactly $r$ runs of $1$'s is determined by choosing the
$2r$ transition boundaries that start and end those runs among the
$m+1=n-1$ gaps around the word. Hence the number is
$\binom{m+1}{2r}=\binom{n-1}{2r}$. Such a word exists precisely when
$2r\leq m+1=n-1$, proving the bound. Equality at the first possible
length $n=2r+1$ leaves only the alternating word.
\end{proof}

\begin{corollary}
For $n\leq4$, every endpoint-reduced nonzero profile has at most one
support block. For $n=5$, the unique two-block zero/nonzero pattern is
\[
(0,v_2,0,v_4,0),
\qquad
v_2,v_4\neq0.
\]
The directions of $v_2$ and $v_4$ are not compared by an adjacent
cup-product equation and may therefore be distinct.
\end{corollary}

\begin{remark}
Theorem~\ref{thm:block-count} separates two different sources of
complexity. The block positions are purely combinatorial, while the
labels $\ell_I\in\PP^1(\Fp)$ are projective. A change of projective
direction can occur only after at least one zero profile entry.
\end{remark}

\section{The one-relator reduction}

The final step in the proof of the three- and four-fold result of
Blumer--Quadrelli uses an elementary correction in the
$E_{2,n+1}$-coordinate \cite[proof of Proposition~5.12]{BlumerQuadrelli}.
We isolate the underlying calculation and use it at arbitrary Massey
length.

For $r\geq1$, write
\[
\UU_{n+1}^{(r)}
=
\{I+(u_{ij})\in\UU_{n+1}:u_{ij}=0\text{ whenever }0<j-i<r\}.
\]
Then
\begin{equation}
[\UU_{n+1}^{(r)},\UU_{n+1}^{(s)}]
\subseteq
\UU_{n+1}^{(r+s)}.
\label{eq:filtration}
\end{equation}
In particular, $\UU_{n+1}^{(n)}=Z(\UU_{n+1})$ and
$\UU_{n+1}^{(n+1)}=\{I\}$.

\begin{lemma}
\label{lem:endpoint-calculus}
Let $n\geq3$ and
\[
K(u)=I_{n+1}+uE_{2,n+1}.
\]
Then:
\begin{enumerate}[label=\textup{(\roman*)}]
\item $K(u)$ has zero first superdiagonal and lies in
$\UU_{n+1}^{(n-1)}$;
\item if $S\in\UU_{n+1}$ has $(1,2)$ entry $a$, then
\[
[S,K(u)]=I_{n+1}+auE_{1,n+1},
\qquad
[K(u),S]=I_{n+1}-auE_{1,n+1};
\]
\item $K(u)$ commutes with every element of $\UU_{n+1}^{(2)}$;
\item if $R\in\UU_{n+1}$ has $(1,2)$ entry zero, then
$[R,K(u)]=I_{n+1}$.
\end{enumerate}
\end{lemma}

\begin{proof}
Part~(i) is immediate. For part~(ii), only the path
$1\to2\to n+1$ contributes to the commutator, giving
\[
[S,K(u)]_{1,n+1}=S_{1,2}u=au.
\]
Equivalently, one may multiply the two elementary factors directly.
The second identity is its inverse. Part~(iii) follows from
\eqref{eq:filtration}, since
\[
[\UU_{n+1}^{(n-1)},\UU_{n+1}^{(2)}]
\subseteq\UU_{n+1}^{(n+1)}=\{I\}.
\]
Part~(iv) is part~(ii) with $a=0$.
\end{proof}

\begin{lemma}
\label{lem:power-preservation}
Let $A\in\UU_{n+1}(\Fp)$ and let $K(u)$ be as above. Let $q=p^k$ be a
finite parameter allowed in the definition of $\Ftwo$: $k\geq1$ if
$p$ is odd and $k\geq2$ if $p=2$. Then
\[
(AK(u))^q=A^q.
\]
If $q=0$, the power term is absent.
\end{lemma}

\begin{proof}
The matrix $K(u)$ has order dividing $p$. By
Lemma~\ref{lem:endpoint-calculus}(ii), the commutator
$[A,K(u)]$ is central and also has order dividing $p$. With our
commutator convention $[A,B]=A^{-1}B^{-1}AB$, the class-two power
formula gives
\[
(AK(u))^q
=
A^qK(u)^q[A,K(u)]^{-\binom{q}{2}}.
\]
The factor $K(u)^q$ is trivial because $p\mid q$. If $p$ is odd, then
$p\mid\binom{q}{2}$. If $p=2$, the defining condition $k\geq2$ gives
$4\mid q$, hence $2\mid\binom{q}{2}$. The commutator factor is therefore
trivial as well.
\end{proof}

The preceding lemma extracts, in a form independent of $n$, the
power-preservation step used in the low-length argument of
Blumer--Quadrelli.

\begin{lemma}
\label{lem:endpoint-correction}
Let $n\geq3$ and suppose
$\langle\alpha_1,\ldots,\alpha_n\rangle$ is defined but is not covered
by Proposition~\ref{prop:endpoint}. Choose lifts
\[
M_s\in\UU_{n+1},
\qquad s\in X_G,
\]
of a Dwyer representation, preserving the prescribed first
superdiagonal. Assume
\[
[M_{z_1},M_{z_2}]=I_{n+1}
\]
and that the Demu\v{s}kin relator evaluates to
\[
M_{x_1}^q[M_{x_1},M_{y_1}]\cdots[M_{x_d},M_{y_d}]
=
I_{n+1}+\delta E_{1,n+1}.
\]
Then the lifts can be modified, without changing their first
superdiagonal and without destroying the exact commuting relation, so
that the Demu\v{s}kin relator also evaluates to the identity.
\end{lemma}

\begin{proof}
If $\alpha_1=0$, the product already contains $0$ by
\cite[Lemma~2.5(a)]{BlumerQuadrelli}. Hence assume $\alpha_1\neq0$.
Because Proposition~\ref{prop:endpoint} does not apply,
\[
\alpha_1(z_1)=\alpha_1(z_2)=0.
\]
Choose $s\in X_G$ with
\[
a:=\alpha_1(s)\neq0.
\]
Then $s\notin\{z_1,z_2\}$. Let $t$ be the Demu\v{s}kin partner of $s$
and replace
\[
M_t\longmapsto M_tK(u),
\qquad
K(u)=I_{n+1}+uE_{2,n+1}.
\]
By Lemma~\ref{lem:endpoint-calculus}(i), this does not alter the
prescribed first superdiagonal.

Consider first the Demu\v{s}kin commutator containing the pair
$\{s,t\}$. With the identities
\[
[X,YZ]=[X,Z][X,Y]^Z,
\qquad
[XY,Z]=[X,Z]^Y[Y,Z],
\]
where $W^Z=Z^{-1}WZ$,
and Lemma~\ref{lem:endpoint-calculus}(iii), replacing the partner
matrix by $M_tK(u)$ multiplies that commutator by
\[
[M_s,K(u)]
\quad\text{or}\quad
[K(u),M_s],
\]
according to the order of $s,t$ in the defining relator. By
Lemma~\ref{lem:endpoint-calculus}(ii), this factor is
\[
I_{n+1}\pm auE_{1,n+1}.
\]
It is central, so no noncentral coordinate of the Demu\v{s}kin relator
is changed.

If $t=z_1$, then
\[
[M_{z_1}K(u),M_{z_2}]
=
[M_{z_1},M_{z_2}]^{K(u)}[K(u),M_{z_2}]
=
I_{n+1},
\]
because $\alpha_1(z_2)=0$ and
Lemma~\ref{lem:endpoint-calculus}(iv) applies. The case $t=z_2$ is
analogous. Thus the exact commuting relator is preserved whenever the
corrected partner is one of the commuting generators.

Finally, if $t=x_1$, then the correction also changes the matrix in the
power term. Lemma~\ref{lem:power-preservation} gives
\[
(M_{x_1}K(u))^q=M_{x_1}^q,
\]
so the power term contributes no additional change. In every case the
only change in the Demu\v{s}kin relator is the central scalar
$\pm au$. Since $a\neq0$, choose $u$ so that this scalar is $-\delta$.
The corrected Demu\v{s}kin relator is then exact.
\end{proof}

\begin{theorem}[One-relator reduction]
\label{thm:one-relator}
Let $G\in\Ftwo$, let $n\geq3$, and suppose
$\langle\alpha_1,\ldots,\alpha_n\rangle$ is defined. If it is not already
covered by endpoint vanishing and some Dwyer representation admits lifts
for which
\[
[M_{z_1},M_{z_2}]=I_{n+1},
\]
then
\[
0\in\langle\alpha_1,\ldots,\alpha_n\rangle.
\]
\end{theorem}

\begin{proof}
Because the chosen matrices lift a homomorphism to $\Ubar_{n+1}$, each
defining relator evaluates in $Z(\UU_{n+1})$. The commuting relator is
exact by hypothesis, and Lemma~\ref{lem:endpoint-correction} removes the
remaining Demu\v{s}kin central defect while preserving both the first
superdiagonal and exact commutation. The corrected matrices therefore
define a continuous homomorphism
\[
G\longrightarrow\UU_{n+1}
\]
with the prescribed first superdiagonal. Proposition~\ref{prop:dwyer}(b)
gives the result.
\end{proof}

\section{The normalized \texorpdfstring{$\UU_6$}{U6} calculation}

We now specialize to length five. The purpose of this section is twofold:
to supply the exact commutator identities needed later, and to correct
the deductions that can be made from them.

Let $C_1,C_2\in\UU_6$. Assume the first superdiagonal of $C_1$ is
\[
(0,\lambda_2,\lambda_3,\lambda_4,0)
\]
and the first superdiagonal of $C_2$ is zero. For $j-i\geq2$ write
\[
(C_1)_{ij}=a_{ij},
\qquad
(C_2)_{ij}=b_{ij},
\]
and use the convention
\[
[A,B]=A^{-1}B^{-1}AB.
\]

\begin{proposition}
\label{prop:u6}
Suppose $[C_1,C_2]$ is central, and write
\[
[C_1,C_2]=I_6+\gamma E_{1,6}.
\]
Then
\begin{align}
b_{13}\lambda_3&=0,
\label{eq:u61}\\
a_{13}b_{35}-a_{35}b_{13}-b_{14}\lambda_4&=0,
\label{eq:u62}\\
-b_{24}\lambda_4+b_{35}\lambda_2&=0,
\label{eq:u63}\\
a_{24}b_{46}-a_{46}b_{24}+b_{36}\lambda_2&=0,
\label{eq:u64}\\
b_{46}\lambda_3&=0,
\label{eq:u65}
\end{align}
and
\begin{equation}
\gamma
=
a_{13}b_{36}
+a_{14}b_{46}
-a_{36}b_{13}
-a_{46}b_{14}.
\label{eq:gamma}
\end{equation}
\end{proposition}

\begin{proof}
With the convention
\[
[A,B]=A^{-1}B^{-1}AB,
\]
centrality gives
\[
C_1C_2=C_2C_1(I_6+\gamma E_{1,6})
       =C_2C_1+\gamma E_{1,6}.
\]
Thus every noncentral entry of $C_1C_2-C_2C_1$ is zero, while its
$(1,6)$ entry is exactly $\gamma$. This avoids any ambiguity from
expanding inverses.

Direct multiplication gives the complete list of potentially nonzero
entries of $C_1C_2-C_2C_1$:
\[
\begin{array}{c|l}
\text{position} & (C_1C_2-C_2C_1)_{ij}\\
\hline
(1,4) & -b_{13}\lambda_3\\
(1,5) & a_{13}b_{35}-a_{35}b_{13}-b_{14}\lambda_4\\
(2,5) & -b_{24}\lambda_4+b_{35}\lambda_2\\
(2,6) & a_{24}b_{46}-a_{46}b_{24}+b_{36}\lambda_2\\
(3,6) & b_{46}\lambda_3\\
(1,6) & a_{13}b_{36}+a_{14}b_{46}
        -a_{36}b_{13}-a_{46}b_{14}.
\end{array}
\]
For example, at position $(1,5)$ one has
\[
(C_1C_2)_{1,5}=a_{15}+b_{15}+a_{13}b_{35},
\]
while
\[
(C_2C_1)_{1,5}=a_{15}+b_{15}+b_{13}a_{35}+b_{14}\lambda_4.
\]
Their difference is
\[
a_{13}b_{35}-a_{35}b_{13}-b_{14}\lambda_4,
\]
which gives \eqref{eq:u62} with the stated commutator convention.
At position $(2,6)$,
\[
(C_1C_2)_{2,6}
=
a_{26}+b_{26}+\lambda_2b_{36}+a_{24}b_{46},
\]
whereas
\[
(C_2C_1)_{2,6}
=
a_{26}+b_{26}+b_{24}a_{46},
\]
which yields \eqref{eq:u64}. Finally,
\[
(C_1C_2-C_2C_1)_{1,6}
=
a_{13}b_{36}+a_{14}b_{46}
-a_{36}b_{13}-a_{46}b_{14},
\]
giving \eqref{eq:gamma}. The remaining three noncentral rows follow by
the same two-factor multiplication. Since no other noncentral entry can be
nonzero under the stated first-superdiagonal assumptions, centrality is
equivalent to \eqref{eq:u61}--\eqref{eq:u65}.
\end{proof}

The following consequences are the ones that will be used. Notice that
``the defect is removable in the normalized model'' is weaker than
``the corresponding Massey product vanishes.''

\begin{corollary}
\label{cor:contiguous}
In the setting of Proposition~\ref{prop:u6}:
\begin{enumerate}[label=\textup{(\alph*)}]
\item If the support of $(\lambda_2,\lambda_3,\lambda_4)$ is
$\{2,3,4\}$, then $\gamma=0$ and hence $[C_1,C_2]=I$.
\item If the support is $\{3,4\}$, then the central defect can be removed
inside the normalized model by changing the free coordinate $b_{36}$.
\item If the support is $\{2,3\}$, then the central defect can be removed
inside the normalized model by changing the free coordinate $b_{14}$.
\item If the support is $\{3\}$, then the central defect can be removed
inside the normalized model by changing $b_{14}$ and/or $b_{36}$.
\end{enumerate}
\end{corollary}

\begin{proof}
For (a), since all three $\lambda_i$ are nonzero,
\eqref{eq:u61} and \eqref{eq:u65} give
$b_{13}=b_{46}=0$. Then \eqref{eq:u62} and \eqref{eq:u63} give
\[
b_{14}=\frac{a_{13}b_{35}}{\lambda_4},
\qquad
b_{24}=\frac{b_{35}\lambda_2}{\lambda_4}.
\]
Equation \eqref{eq:u64} gives
\[
b_{36}=\frac{a_{46}b_{24}}{\lambda_2}
=\frac{a_{46}b_{35}}{\lambda_4}.
\]
Substitution into \eqref{eq:gamma} yields
\[
\gamma
=
a_{13}\frac{a_{46}b_{35}}{\lambda_4}
-
a_{46}\frac{a_{13}b_{35}}{\lambda_4}
=0.
\]

For (b), $\lambda_2=0$ and $\lambda_3,\lambda_4\neq0$. The equations
give
\[
b_{13}=b_{24}=b_{46}=0,
\qquad
b_{14}=\frac{a_{13}b_{35}}{\lambda_4},
\]
while $b_{36}$ remains free. Hence
\[
\gamma
=
a_{13}\left(
b_{36}-\frac{a_{46}b_{35}}{\lambda_4}
\right).
\]
If $a_{13}=0$, the defect is already zero; otherwise choose
$b_{36}=\frac{a_{46}b_{35}}{\lambda_4}$.

For (c), $\lambda_4=0$ and $\lambda_2,\lambda_3\neq0$. The equations
give
\[
b_{13}=b_{35}=b_{46}=0,
\qquad
b_{36}=\frac{a_{46}b_{24}}{\lambda_2},
\]
while $b_{14}$ remains free. Thus
\[
\gamma
=
a_{46}\left(
\frac{a_{13}b_{24}}{\lambda_2}-b_{14}
\right),
\]
which can be killed by choosing $b_{14}$ unless $a_{46}=0$, in which
case it is already zero.

For (d), $\lambda_2=\lambda_4=0$ and $\lambda_3\neq0$. Then
$b_{13}=b_{46}=0$ and
\[
\gamma=a_{13}b_{36}-a_{46}b_{14}.
\]
If $(a_{13},a_{46})=(0,0)$ the defect vanishes; otherwise the two free
coordinates can be chosen to make $\gamma=0$.
\end{proof}

\begin{remark}
\label{rem:compatibility}
	For parts \textup{(b)}--\textup{(d)} of Corollary~\ref{cor:contiguous},
	the normalized corrections identify coordinates in the commuting pair that
	can be adjusted to eliminate the central commuting defect. Extending these
	corrections to vanishing of the corresponding Massey products requires
	compatibility with the full Demu\v{s}kin presentation: the induced change in
	the Demu\v{s}kin partner must be controlled so that the defining relator
	remains satisfied modulo the center. In part \textup{(a)}, no such additional
	compatibility step is needed, since centrality already forces $[C_1,C_2]=I$.
\end{remark}

\section{The length-five theorem}

We now combine the support-block geometry, the normalized calculation, and
Theorem~\ref{thm:one-relator}.

\begin{theorem}[Full interior support]
\label{thm:full-support}
Let $G\in\Ftwo$. Let
\[
\langle\alpha_1,\ldots,\alpha_5\rangle
\]
be defined. If
\[
v_2,v_3,v_4\neq0,
\]
then
\[
0\in\langle\alpha_1,\ldots,\alpha_5\rangle.
\]
\end{theorem}

\begin{proof}
If the endpoint theorem applies, there is nothing to prove, so assume
\[
v_1=v_5=0.
\]
By Proposition~\ref{prop:projective}, the three nonzero vectors
$v_2,v_3,v_4$ lie on a common projective line
$\ell\in\PP^1(\Fp)$.

If $\ell$ is one of the coordinate axes, then one fixed coordinate
$\alpha_h(z_i)$ vanishes for every $h=1,\ldots,5$, because the endpoint
vectors are also zero. Proposition~\ref{prop:hyperplane} then gives
vanishing. We may therefore assume that both coordinates of $\ell$ are
nonzero.

Choose a Dwyer representation and arbitrary lifts of the two commuting
generators,
\[
C_1,C_2\in\UU_6.
\]
Write
\[
\lambda_h=\alpha_h(z_1),
\qquad h=2,3,4.
\]
Since $\ell$ is not a coordinate axis, all $\lambda_h$ are nonzero and
there is a single $\kappa\in\Fp^\times$ such that
\[
\alpha_h(z_2)=\kappa\lambda_h,
\qquad h=2,3,4.
\]
The endpoint coordinates are zero as well.

Choose an integer representative of $\kappa$ and set
\[
D=C_2C_1^{-\kappa}.
\]
The first superdiagonal of $D$ is zero. Since the commuting relator is
trivial in $\Ubar_6$, the commutator $[C_1,C_2]$ is central. As $C_1$
commutes with its own powers, the commutator $[C_1,D]$ is the same central
element. We are therefore in the setting of
Proposition~\ref{prop:u6}, with
$\lambda_2,\lambda_3,\lambda_4\neq0$. By
Corollary~\ref{cor:contiguous}(a),
\[
[C_1,D]=I.
\]
Hence $C_1$ commutes with $D$, and therefore also with
\[
C_2=DC_1^\kappa.
\]
Thus the commuting relator is exact.

Theorem~\ref{thm:one-relator} removes the remaining Demu\v{s}kin central
defect. Dwyer's criterion then gives the result.
\end{proof}

\begin{corollary}
\label{cor:length-five-reduction}
Let $G\in\Ftwo$. Any defined five-fold Massey product not already covered
by the endpoint or coordinate-hyperplane vanishing theorems and not
covered by Theorem~\ref{thm:full-support} has one of the following
interior supports:
\[
\{2\},\quad
\{3\},\quad
\{4\},\quad
\{2,3\},\quad
\{3,4\},\quad
\{2,4\}.
\]
The last support is the unique two-block pattern.
\end{corollary}

\begin{remark}
Corollary~\ref{cor:contiguous} gives normalized commuting-defect
corrections for $\{3\}$, $\{2,3\}$ and $\{3,4\}$. To upgrade those
calculations to complete vanishing theorems, one must prove that the
chosen correction can be accompanied by corrections of the
Demu\v{s}kin partners without reintroducing a commuting defect.
The singleton supports $\{2\}$ and $\{4\}$ require a separate
calculation from the outset.
\end{remark}

\section{The separated two-block profile}

In length five the only two-block profile is
\[
(v_1,\ldots,v_5)=(0,v_2,0,v_4,0),
\qquad
v_2,v_4\neq0.
\]
It carries an ordered pair of directions
\[
(\ell_L,\ell_R)\in\PP^1(\Fp)\times\PP^1(\Fp).
\]
If $\ell_L\neq\ell_R$, then there are
\[
(p+1)p
\]
ordered distinct-direction projective types.

The support-block formalism does not compare $\ell_L$ and $\ell_R$: the
zero at position $3$ removes the adjacent determinant equation that would
otherwise force equality. This is the first genuinely new projective
phenomenon in the hierarchy.

For a Dwyer representation
\[
\bar\rho:G\to\Ubar_6
\]
choose lifts of the generator matrices. After all relators are imposed
modulo the center, write the two residual central defects as
\[
\omega=(\delta,\gamma)\in\Fp^2,
\]
with $\delta$ corresponding to the Demu\v{s}kin relator and $\gamma$ to
the commuting relator. Theorem~\ref{thm:one-relator} shows that the
essential target is one-dimensional: it is enough to kill $\gamma$.
Once $[z_1,z_2]$ is exact, $\delta$ can be removed automatically.

This reduces the remaining two-block problem to the following question.

\begin{question}
\label{q:two-island}
Let a defined five-fold product have profile
\[
(0,v_2,0,v_4,0),
\qquad
v_2,v_4\neq0.
\]
Can every Dwyer representation be moved, without changing its first
superdiagonal and while keeping the Demu\v{s}kin relator trivial modulo
the center, to one for which
\[
[M_{z_1},M_{z_2}]=I?
\]
\end{question}

A positive answer would settle the distinct-direction two-island case
and, together with the remaining one-block compatibility calculations,
would prove five-fold vanishing for $\Ftwo$.

\section{Arithmetic subfamilies from local fields}

The group-theoretic results above admit a natural arithmetic
specialization. Let $p$ be odd and let $K/\mathbb Q_p$ be a finite
extension containing $\mu_p$. Its maximal pro-$p$ Galois group
\[
D_K=G_K(p)
\]
is a Demu\v{s}kin group with a standard presentation
\[
D_K
=
\left\langle
x_1,y_1,\ldots,x_d,y_d
\ \middle|\
x_1^{q_K}[x_1,y_1]\cdots[x_d,y_d]=1
\right\rangle_{\widehat p},
\]
where $q_K$ is the Demu\v{s}kin invariant \cite{NSW}. Choose distinct
generators $z_1,z_2$ which do not form a Demu\v{s}kin pair and set
\[
G_K(z_1,z_2)
=
D_K/
\overline{\langle\!\langle[z_1,z_2]\rangle\!\rangle}.
\]
Then $G_K(z_1,z_2)\in\Ftwo$.

\begin{corollary}
\label{cor:local}
For every $G_K(z_1,z_2)$ as above:
\begin{enumerate}[label=\textup{(\roman*)}]
\item for every $n\geq3$, exactification of the commuting relator in a
Dwyer lift is sufficient for vanishing of a defined $n$-fold Massey
product;
\item every defined five-fold Massey product with
$v_2,v_3,v_4\neq0$ contains $0$.
\end{enumerate}
\end{corollary}

\begin{proof}
Apply Theorems~\ref{thm:one-relator} and \ref{thm:full-support}.
\end{proof}

This gives a field-parametrized arithmetic subfamily of the
Blumer--Quadrelli variations on which the support-block and one-relator
reductions apply directly. The restriction to odd $p$ here is only to
keep the displayed local Demu\v{s}kin presentation in its standard form;
the group-theoretic theorems for $\Ftwo$ apply to every prime.

\section{Consequences and further directions}

\subsection{Block complexity versus length}

Theorem~\ref{thm:block-count} suggests organizing the all-length problem
by block complexity. Define the $B$-block vanishing property to mean that
every defined Massey product with at most $B$ support blocks contains $0$.
The known all-length results give vanishing on coordinate hyperplanes and
at endpoints, while the first possible $B$-block profile appears in
length $2B+1$.

This hierarchy separates two sources of higher-length complexity:
\begin{enumerate}[label=\textup{(\roman*)}]
\item more possible support blocks;
\item more independent projective directions attached to those blocks.
\end{enumerate}
The first two-block geometry occurs in length five; the first three-block
geometry occurs in length seven.

\subsection{One-relator strategy}

Theorem~\ref{thm:one-relator} gives a complementary organization of the
matrix problem. At every length, one may first solve only
\[
[M_{z_1},M_{z_2}]=I.
\]
The Demu\v{s}kin defect can then be removed afterward by
Lemma~\ref{lem:endpoint-correction}. This removes one row from the
correction problem and avoids simultaneously solving two independent
central equations.

For length five, the remaining tasks are therefore:
\begin{enumerate}
\item prove full-presentation compatibility for the normalized corrections
in supports $\{3\}$, $\{2,3\}$ and $\{3,4\}$;
\item treat the singleton supports $\{2\}$ and $\{4\}$;
\item solve the separated support $\{2,4\}$, first when the two projective
directions coincide and then when they are distinct.
\end{enumerate}

The same organization persists in higher length. The support-block
hierarchy controls which first-superdiagonal patterns can occur, while the
one-relator reduction identifies exact commutation as the remaining target
within this unitriangular reduction strategy.

Taken together, the support-block hierarchy and the one-relator reduction
provide a uniform framework for organizing the higher Massey problem in
\(\mathcal F_2\), both in length five and beyond.

\end{document}